\documentclass[12pt,reqno,a4paper]{amsart}
\usepackage{amsmath,amssymb,amsfonts,amscd}
\usepackage[mathscr]{eucal}
\usepackage{color}
\usepackage{hyperref}

\date{1 (14) July  2025}

\author{Theodore~Voronov}
\address{Department of Mathematics, University of Manchester, Manchester, M13 9PL, UK}
\email{theodore.voronov@manchester.ac.uk}

\title[Graded and shifted notions]{On   graded and shifted notions, and thick morphisms}

\newtheorem{theorem}{Theorem}
\newtheorem{proposition}{Proposition}

\theoremstyle{definition}
\newtheorem{definition}{Definition}
\newtheorem{remark}{Remark}

\def\co{\colon\thinspace}

\DeclareMathOperator{\fun}{\mathit{C^{\infty}}}
\DeclareMathOperator{\funn}{\mathbf{C^{\infty}}}
\DeclareMathOperator{\pfunn}{\mathbf{\Pi\!C^{\infty}}}

 \DeclareMathOperator{\sign}{sgn}

\newcommand{\der}[2]{{\frac{\partial {#1}}{\partial {#2}}}}

\newcommand{\lder}[2]{{\partial {#1}/\partial {#2}}}

\newcommand{\var}[2]{{\frac{\delta {#1}}{\delta {#2}}}}

\newcommand{\Z}{{\mathbb Z_{2}}}
\newcommand{\ZZ}{{\mathbb Z}}

\newcommand{\p}{\partial}

\newcommand{\w}{{\mathbf{w}}}

\renewcommand{\a}{\alpha}
\renewcommand{\b}{\beta}
\newcommand{\e}{\varepsilon}
\newcommand{\s}{\sigma}
\newcommand{\f}{{\varphi}}

\newcommand{\F}{{\Phi}}

\newcommand{\la}{{\lambda}}
\newcommand{\x}{{\xi}}

\newcommand{\itt}{{\tilde \imath}}

\newcommand{\at}{{\tilde a}}
\newcommand{\bt}{{\tilde b}}

\newcommand{\ut}{{\tilde u}}

\newcommand{\xt}{{\tilde x}}
\newcommand{\yt}{{\tilde y}}

\newcommand{\Qt}{{\tilde Q}}

\newcommand{\lat}{{\tilde \lambda}}
\newcommand{\xit}{{\tilde \xi}}

\DeclareMathOperator{\Vect}{\mathrm{Vect}}

\newcommand{\lsch}{{[\![}}
\newcommand{\rsch}{{]\!]}}

\newcommand{\Sinf}{S_{\infty}}
\newcommand{\Pinf}{P_{\infty}}
\newcommand{\Linf}{L_{\infty}}

\newcommand{\tto}{{\linethickness{2pt}
		  \,\begin{picture}(1,0)
                   \put(0,0.26){\line(1,0){0.95}}
                   \put(0,0){$\boldsymbol{\rightarrow}$}
                  \end{picture}
                  }\,
}

\newcommand{\oto}{{\linethickness{0.5pt}
		  \,\begin{picture}(1,0)
		  \put(0.07,0.175){\line(0,1){0.2}}
                   \put(-0.01,0){$\boldsymbol{\Rightarrow}$}
                  \end{picture}
                  }\,
}

\newcommand{\void}{\varnothing}

\newcommand{\yp}{\boldsymbol{y}}
\newcommand{\xp}{\boldsymbol{x}}

\newcommand{\T}{\mathrm{T}}

\begin{document}
\begin{abstract}
We consider the notions of $L_{\infty}$-, $P_{\infty}$-,  and $S_{\infty}$-algebras (including ``shifted'' versions) in  the $\mathbb{Z}_2 \times \mathbb{Z}$-graded setting. We also consider thick (microformal) morphisms and show   how they work in such graded context.  

In particular, we show that a ``shifted $S_{\infty}$-thick morphism'' (which  we
introduce here) induces an $L_{\infty}$-morphism of shifted $S_{\infty}$-structures. The same holds for ``shifted $P_{\infty}$-thick morphisms'' and shifted $P_{\infty}$-structures, respectively.
\end{abstract}
\maketitle
\tableofcontents
 \section{Introduction.}
 
Standard approach to homotopy algebras  such as $\Linf$-, $\Pinf$-,  and $\Sinf$-algebras (homotopy generalizations of Lie, Poisson, and odd Poisson structures, respectively) uses $\ZZ$-grading, so that signs in the formulas are determined by the degrees. On the other hand, in supergeometry, ``parity'' and ``degree'' do not have to agree (for example, there are even $1$-forms and odd $2$-forms). From physical viewpoint, linking parities with degrees is also unnatural because excludes theories with fermions. 

The author has always maintained that parity\,---\,responsible for signs\,---\,should be kept separate from any  $\ZZ$-gradings\,---\,used for counting. In a concrete situation, there can be no $\ZZ$-grading at all, or several $\ZZ$-gradings (which   may  be combined  one way or another). 

One purpose of this note is to present  the notions of $\Linf$-, $\Pinf$-,  and $\Sinf$-algebras in the $\Z \times \ZZ$-graded (``graded superized'')  setting, as opposed to  the  conventional    $\ZZ$-graded      and   my   favorite  $\Z$-graded   (``purely super'') settings. There is no pretence for originality in the definitions, which I believe should be folklore. Still it is worth putting them down. 
Our consideration   includes     ``shifted'' notions (see on them~\cite{pridham:outline2018} and \cite{behrend-peddie-xu:2023}). There is an overlap with parts of~\cite{pridham:outline2018}, but   our    exposition is simpler because it is based on the more elementary differential-geometric language. 

Another goal is to show  how  thick morphisms of supermanifolds~\cite{tv:nonlinearpullback, tv:microformal}   work in the presence of a $\ZZ$-grading. In particular, we show that a ``shifted $\Sinf$-thick morphism'' (a notion we introduce here) induces an $\Linf$-morphism of shifted $\Sinf$-structures. (It holds similarly for the shifted $\Pinf$ case.)

\emph{Notation}: we denote \emph{parity} ($\Z$-grading) by the tilde over a symbol and may use different names and notations for a $\ZZ$-grading, such as $\w$ for ``weight'' and $\deg$ for ``degree''\,.

\section{Versions of $\Linf$-algebras in the purely super setting.}

\begin{definition}\label{def.linf}
  An \emph{$\Linf$-algebra} is a super vector space (or a module over a commutative superalgebra) $L$ with a sequence of multilinear operations
  \begin{equation}
    [-,\ldots,-]\co \underbrace{L\times \ldots \times L}_{\text{$n$ times}} \to L \quad \text{(for $n=0,1,2,\ldots $)}
  \end{equation}
   of   parities $n\mod 2$,  
   which are antisymmetric in the supersense and satisfy  higher Jacobi identities in the form
  \begin{equation}\label{eq.hjac}
    \sum_{r+s=n}\sum_{\text{shuffles}}  (-1)^{\b}  [[x_{\s(1)},\ldots,x_{\s(r)}], \ldots , x_{\s(r+s)}]=0
  \end{equation}
  for all~$n=0,1,2,3,...$, where $(-1)^{\b}=(-1)^{rs}\sign \s (-1)^{\a}$ and $(-1)^{\a}$  is the Koszul sign.
\end{definition}

Comments: this is a super setting version    (meaning $\ZZ$ is replaced by $\Z$) of the Lada-Stasheff~\cite{lada:stasheff} definition apart from including a $0$-bracket $[\void]\in L_0$. The latter is referred to as `background'~\cite{zwiebach:93csft} or `curvature' and structures with it are sometimes called `curved'. We systematically suppress this adjective and all algebras that we consider are `curved' by default.
`Linearity' in the case of modules means
\begin{equation}\label{eq.lin}
  [\la x_1,\ldots,x_n]= (-1)^{\lat n}\la [x_1,\ldots,x_n]\,, \quad
  [x_1 \la ,x_2,\ldots,x_n]=[x_1 , \la x_2,\ldots,x_n]\,,
\end{equation}
etc.

\begin{definition}\label{def.oddlinf}
  An  \emph{odd $\Linf$-algebra} is a super vector space (or a module over a commutative superalgebra) $V$ with a sequence of multilinear operations
  \begin{equation}
    [-,\ldots,-]\co \underbrace{V\times \ldots \times V}_{\text{$n$ times}} \to V \quad \text{(for $n=0,1,2,\ldots $)}\,,
  \end{equation}
   which are odd, symmetric in the supersense and satisfy  higher Jacobi identities  in the form
  \begin{equation}\label{eq.hjacod}
    \sum_{r+s=n}\sum_{\text{shuffles}}  (-1)^{\a}  [[u_{\s(1)},\ldots,u_{\s(r)}], \ldots , u_{\s(r+s)}]=0
  \end{equation}
  for all~$n=0,1,2,3,...$, where $(-1)^{\a}$  is the Koszul sign.
\end{definition}

These notions are equivalent via parity reversion. If $L$ is an $\Linf$-algebra in the sense of Definition~\ref{def.linf}, then $V=\Pi L$ is an odd $\Linf$-algebra in the sense of Definition~\ref{def.oddlinf}, and conversely, where the brackets in $L$ and $V=\Pi L$ are related by
\begin{equation}\label{eq.brack}
  \Pi [x_1,\ldots,x_n] =(-1)^{\xt_1(n-1)+\xt_2(n-2)+\ldots+\xt_{n-1}} [\Pi x_1,\ldots,\Pi x_n]\,,
\end{equation}
where $x_i\in L$. One can check that~\eqref{eq.brack} turns antisymmetry into symmetry on the parity-reversed space  and conversely, and that one version of the Jacobi identities is transformed into another.\footnote{In~\cite{tv:higherder, tv:higherderarb}, we called the notion specified by Definition~\ref{def.oddlinf}, just an ``$\Linf$-algebra''; but here we aim at clarifying the terminology.}

A ``combined'' version can be given by
\begin{definition}\label{def.shlinf}
Let $\e=0,1$.
  An  \emph{$\e$-shifted $\Linf$-algebra} is a super vector space (or a module over a commutative superalgebra) $V$ with a sequence of multilinear operations
  \begin{equation}
    [-,\ldots,-]\co \underbrace{V\times \ldots \times V}_{\text{$n$ times}} \to V \quad \text{(for $n=0,1,2,\ldots $)}\,,
  \end{equation}
of parities $\e(n+1)+n= n(\e + 1) + \e $ making $V$ an $\Linf$-algebra in the sense of Definition~\ref{def.linf} for $\e=0$ and an odd $\Linf$-algebra in the sense of Definition~\ref{def.oddlinf} for $\e=1$.
\end{definition}

The geometric description can be given as follows~\cite{tv:higherder}. 

Let $V$ be a vector superspace. It can be treated also as a supermanifold by considering coordinates of an even vector relative to some basis as independent variables of the corresponding parity. (So every basis $e_i$ in $V$ as a $\Z$-graded vector space  gives rise to variables $x^i$, where $\tilde{x}^i=\tilde{e_i}$, regarded as global coordinates on $V$ as a supermanifold, so that the expression $\xp=x^ie_i$ representing a `running even vector' in $V$ is invariant under changes of bases.) Similarly for the parity-reversed superspace $\Pi V$.

Define $i\co V\to \Vect(\Pi V)$ by $i(u)=i_u:=(-1)^{\ut}u^i\lder{}{\x^i}$ if $u=u^ie_i$. Here $\x^i$ are (left) coordinates on $\Pi V$ corresponding to a basis $e_i\in V$. This is an odd monomorphism, giving an odd isomorphism between vectors in $V$ and `constant' vector fields on $\Pi V$. (A similar map $V\to \Vect(V)$ can be seen as the  identical inclusion.)

\begin{proposition}\label{prop.homvf}
  An    $\e$-shifted $\Linf$-algebra  structure  on a superspace $V$ is specified by a formal  homological vector field $Q\in \Vect(\Pi^{1+\e}V)$, i.e.    $\Qt=1$  and $Q^2=0$.
  For $\e=1$, the brackets are given by
  \begin{equation}\label{eq.odbrQ}
    [u_1,\ldots,u_n]_Q = [\ldots[Q, u_1],\ldots, u_n](0)\,.
  \end{equation}
  For $\e=0$, the brackets are given by
  \begin{equation}\label{eq.altbrQ}
    i_{[u_1,\ldots,u_n]_Q}=(-1)^{\ut_1(n-1)+\ldots +\ut_{n-1}}[\ldots[Q,i_{u_1}],\ldots,i_{u_n}](0)\,.
  \end{equation}
\end{proposition}

We always mean by $Q$ a formal vector field, but will suppress the adjective.

\section{Versions of $\Linf$-algebras with a $\ZZ$-grading.}

It is usual to think about  signs as  coming from   ``degrees'', i.e.  a $\ZZ$-grading. We prefer to keep it separate. In the sequel we deal with $\Z\times \ZZ$-graded vector spaces and algebras. If necessary, we may write them as $V=(V^n)$, $n\in \ZZ$, and $V^n=V^n_0\oplus V^n_1$ (where $0,1\in \Z$). 

We refer to $\Z$-grading as \emph{parity} and $\ZZ$-grading  as \emph{weight}. 

Parity  is   responsible for signs and weight is a  counting tool.

\subsection{On grading and shifts.}
There are independent functors $\Pi$ (parity reversion) and $\T=[1]$ (shift). We define $(V[1])^n=V^{n+1}$ and $(V[s])^n=V^{n+s}$ for any $s\in \ZZ$. Let $\w(u)=n$ stand  for the weight of $u\in V^n$. Weight in $V[s]$ will be denoted $\w^{[s]}$. Then $\w^{[s]}(u)=n-s$ if $\w(u)=n$.  

A shifted space $V[n]$ can be regarded as another copy of $V$, with the notation $u[s]\in V[s]$ if $u\in V$. If this notation is used, there is no particular need for distinguishing $\w^{[s]}$. We can simply write $\w(u[s])$ and  $\w(u[s])=\w(u)-s$.  If $e_i$ is a basis in $V$, then $e_i[s]$ is a basis in $V[s]$. 

When we treat a vector space as a graded supermanifold\footnote{We   often omit the adjective ``graded''   unless we want to emphasize   a $\ZZ$-grading. Likewise, we can suppress the prefix ``super'' if it does not lead to a confusion.}, the general principle is:   
\begin{center}
\emph{``points'' are always even and of weight zero}. 
\end{center}
A choice of basis in $V$ gives rise to coordinates in $V$ as a supermanifold; if we stick to left coordinates, these are indeterminates $x^i$ such that the expression $x^ie_i$ is invariant under changes of basis, is even and of weight zero. Hence $\xt^i=\widetilde{e_i}=\itt$ and $\w(x^i)=-\w(e_i)$.

Likewise, for a shifted space $V[s]$ regarded as a   supermanifold, coordinates associated with a basis $e_i$ in $V$ will be $y^i$, $\w(y^i)=\w(x^i)+s$  (if $x^i$ are coordinates for $V$). So that   $\yp[s]=y^i e_i[s]$ is a `running  vector of weight zero' in $V[s]$ or  alternatively   $\yp=y^i e_i$ is a `running  vector of weight $s$' in $V$. 

From this viewpoint, treating an ordinary (ungraded)  vector space $V$ as a graded manifold with  coordinates of weight $+1$, which is a usual practice, means replacing it by the graded manifold $V[1]$, in the present notation. The same applies to vector bundles: writing $E[1]$ for an ordinary vector bundle $E$   means treating it as a graded manifold  with all fiber coordinates assigned weight $+1$; if $E$ is already a vector bundle in the category of graded manifolds, then $E[s]$ means a new vector bundle with the same transition functions and where all fiber coordinates have weights advanced by $s$.

\subsection{Two  graded versions of $\Linf$-algebras.}

Below ``graded'' means $\Z\times \ZZ$-graded. Commutators, commutativity, linearity, etc. are concerned only with the $\Z$-part of the grading.

\begin{definition}\label{def.grlinf}
  Let $V$ be a graded vector space (or   module over a  commutative superalgebra). An (even) \emph{$\Linf$-algebra} (resp., \emph{odd $\Linf$-algebra}) structure in $V$ is defined as in Definition~\ref{def.linf} (resp., in Definition~\ref{def.oddlinf}) with an additional requirement that the $n$th bracket
  \begin{equation}
    [-,\ldots,-]\co \underbrace{V\times \ldots \times V}_{\text{$n$ times}} \to V
  \end{equation}
  has weight $2-n$, for all $n$.
\end{definition}

For the even version, this will give  the Lada--Stasheff definition~\cite{lada:stasheff} if parity equals weight mod $2$  and there is no $0$-bracket.

\begin{proposition}\label{prop.homvfgrad}
  A structure of an \emph{$\Linf$-algebra} (resp.,  odd $\Linf$-algebra) in  a graded vector space $V$ in the sense of Definition~\ref{def.grlinf} is equivalent to a homological vector field
  \begin{equation}\label{eq.Q}
    Q\in \Vect(\Pi V[1]) \quad \text{(resp., $Q\in \Vect(V[1])$)}
  \end{equation}
of weight $+1$.  The brackets are given by the same formulas~\eqref{eq.altbrQ} or~\eqref{eq.odbrQ}. (In~\eqref{eq.odbrQ}, one has to write $i_u=u[1]$ instead of $u$.)
\end{proposition}

For a graded vector space, the  maps $i\co V\to \Vect(\Pi V[1])$ and $i\co V\to \Vect(V[1])$ defined by $i(u)=i_u:=(-1)^{\ut}u^i\lder{}{\x^i}$ and $i(u)=i_u:= u^i\lder{}{y^i}$, respectively, where $\w(\x^i)=-\w(e_i)+1$, $\w(y^i)=-\w(e_i)+1$, have both weight $-1$. (The map $i\co V\to \Vect(\Pi V[1])$ is odd, the map $i\co V\to \Vect(V[1])$ is even.)

 \subsection{Shifted $\Linf$-algebras.}

Now towards the ``shifted'' version. Consider a graded vector space (or a module) $V$. Consider $V[s]$ and $\Pi V[s]$. We will use the same notation, $u\mapsto i_u$, for the linear maps
\begin{equation*}
  i\co V\to \Vect(V[s])  \quad \text{and} \quad   i\co V\to \Vect(\Pi V[s])
\end{equation*}
defined, respectively, as $u\mapsto u[s]\in V[s]$ and $u\mapsto (-1)^{\ut}\Pi u[s]\in \Pi V[s]$ followed by the   inclusions $V[s]\hookrightarrow \Vect(V[s])$  and $\Pi V[s]\hookrightarrow \Vect(\Pi V[s])$ as constant vector fields. Both   $i\co V\to \Vect(V[s])$ and $i\co V\to \Vect(\Pi V[s])$ have weights $-s$, the first map is even, the second map is odd.

In coordinates, $i_u=u^i\der{}{y^i}$ for $i\co V\to \Vect(V[s])$ and  $i_u=(-1)^{\ut}u^i\der{}{\x^i}$ for $i\co V\to \Vect(\Pi V[s])$. Here $\w(y^i)=\w(\x^i)=w^i+s$, where $w^i=-\w(e_i)$ are the weights of coordinates in $V$, and $\yt^i=\itt$, $\xit^i=\itt+1$.

Fix $\e\in \Z$ and $k\in \ZZ$.  An ``$(\e,k)$-shifted'' $\Linf$-structure on $V$ is essentially an $\Linf$-structure on $V[-k]$, even for $\e=1$ and odd for $\e=0$. The brackets on $V$ are induced from the brackets on  $V[-k]$, by $u\mapsto u[-k]$. In more detail it is given by the following definition (where the only new ingredient is calculation of   weights).

\begin{definition}\label{def.shgrlinf}
An   \emph{$(\e,k)$-shifted $\Linf$-algebra}   structure in $V$ is a sequence of multilinear operations of weights $2-n+k(n-1)$ and parities  $\e(n+1)+n$
 \begin{equation}\label{eq.shbrack}
    [-,\ldots,-]\co \underbrace{V\times \ldots \times V}_{\text{$n$ times}} \to V\,,
  \end{equation}
where $n=0,1,2,\ldots\ $, which for $\e=0$ are antisymmetric and satisfy higher Jacobi identities in the form~\eqref{eq.hjac}, and for $\e=1$ are symmetric and satisfy higher Jacobi identities in the form~\eqref{eq.hjacod}.
\end{definition}

\begin{proposition}\label{prop.shandsh}
An  $(\e,k)$-shifted $\Linf$-algebra   structure in $V$ is the same as an $\e$-shifted $\Linf$-structure on $V[-k]$, if the brackets in   $V$ and $V[-k]$ are given by the same formulas. 
\end{proposition}
\begin{proof}
  Since a  shift  of $\ZZ$-grading does not affect symmetry properties or Jacobi identities (which remain the same), all what is needed, is to check the weights of the brackets. Indeed, if we consider weights $\w^{[-k]}$, i.e. in $V[-k]$,   for $[u_1,\ldots,u_n]$ given by~\eqref{eq.shbrack},  we will obtain 
\begin{multline*}
  \w^{[-k]}\left([u_1,\ldots,u_n]\right)=\w\left([u_1,\ldots,u_n]\right)+k=\w(u_1)+\ldots + \w(u_n) + 2-n+k(n-1)+k=\\
  \w^{[-k]}(u_1)+\ldots + \w^{[-k]}(u_n) +nk
  + 2-n+k(n-1)+k=\\
  \w^{[-k]}(u_1)+\ldots + \w^{[-k]}(u_n)+2-n\,,
\end{multline*}
i.e. weight    $2-n$ for the $n$-bracket     in $V[-k]$, as desired.
\end{proof}


Note again that linearity, symmetry and antisymmetry are controlled only by parity and not by weight.
For $k=0$,  we return to   $\Linf$-algebras or odd $\Linf$-algebras in the sense of Definition~\ref{def.grlinf}  if $\e=0$ or $\e=1$, respectively.

If one wants to link parity with weight, then in the definition of a  shifted structure $\e\in \Z$ has to be set to $k\mod 2$.

Again, there is a geometric description: 

\begin{proposition}\label{prop.homf}
  An    $(\e,k)$-shifted $\Linf$-algebra  structure  on  $V$ is equivalent to a   homological vector field of weight $+1$
  \begin{equation}\label{eq.Qshgrlinf}
    Q\in \Vect(\Pi^{1+\e}V[1-k])\,.
  \end{equation}
 For $\e=0$, the brackets are given by
  \begin{equation}\label{eq.shaltbrQ}
    i_{[u_1,\ldots,u_n]_Q}=(-1)^{\ut_1(n-1)+\ldots +\ut_{n-1}}[\ldots[Q,i_{u_1}],\ldots,i_{u_n}](0)\,.
  \end{equation}
  For $\e=1$, the brackets are given by
  \begin{equation}\label{eq.shodbrQ}
    i_{[u_1,\ldots,u_n]_Q} = [\ldots[Q, i_{u_1}],\ldots, i_{u_n}](0)\,.
  \end{equation}
 Here $u_1,\ldots,u_n\in V$.
\end{proposition}
\begin{proof}
In view of the geometric description for the ungraded case by Proposition~\ref{prop.homvf}, what remains is to check the weights. 
Suppose the $n$-bracket is given by~\eqref{eq.shaltbrQ} or~\eqref{eq.shodbrQ}. Then in the RHS, one has the weight $1+\w(u_1)-(1-k)+ \ldots+ \w(u_n)-(1-k)=\w(u_1) + \ldots+ \w(u_n) + n(k-1)+1$; in the LHS, the weight is $\w([u_1,\ldots,u_n]_Q)-(1-k)$. Hence
  \begin{multline*}
    \w([u_1,\ldots,u_n]_Q)=\w(u_1) + \ldots+ \w(u_n)+  (1-k)+n(k-1)+1=\\
     \w(u_1) + \ldots+ \w(u_n)+ 2-n +k(n-1)\,,
  \end{multline*}
 as claimed.
\end{proof}

(Otherwise, one can apply Proposition~\ref{prop.shandsh} to replace an $(\e,k)$-shifted $\Linf$   structure  on  $V$ by an  $\e$-shifted $\Linf$   structure  on  $V[-k]$ and then apply Proposition~\ref{prop.homvfgrad} to obtain a homological vector field of  weight $+1$ on $\Pi^{1+\e}V[-k][1]=\Pi^{1+\e}V[1-k]$.)

\section{Graded and shifted versions of $\Pinf$- and $\Sinf$-algebras.}
\label{sec.shpinfsinf}

\subsection{Shifted (anti)cotangent bundle.}
Suppose $M$ is a graded (super)manifold. Then $T^*M$ is also a graded (super)manifold w.r.t. induced grading: if $x^a$ are local coordinates on $M$ with weights $\w(x^a)=w^a\in\ZZ$, then the fiber coordinates $p_a$ have weights $\w(p_a)=-w^a$. (Changes of coordinates $x^a,p_a$ on $T^*M$ respect weights.) The same holds for the anticotangent bundle $\Pi T^*M$. Local coordinates are $x^a,x^*_a$ of weights $\w(x^a)=w^a$, $\w(x^*_a)=-w^a$. The canonical Poisson bracket  on $T^*M$ and Schouten bracket  on $\Pi T^*M$ are both of weight $0$.  (Indeed, one may recall the explicit formulas, where, in each term, the partial derivative of one argument with respect to a coordinate is accompanied by the partial derivative of the other argument with respect to the conjugate momentum or anti-momentum of the opposite weight.)

As vector bundles, $T^*M$ and   $\Pi T^*M$ can be endowed with another grading, which is just the degree in fiber coordinates $p_a$ or $x^*_a$. More generally, one can assign degree $s$ to all fiber coordinates, for an arbitrary fixed $s\in \ZZ$. We denote by
\begin{equation*}
  T^*M[s]\quad \text{and} \quad \Pi T^*M[s]
\end{equation*}
the corresponding bundles with (total) weight counted as $\w^{[s]}(x^a)=\w(x^a)=w^a$, $\w^{[s]}(p_a)=-w^a+s$, and $\w^{[s]}(x^*_a)=-w^a+s$. There remains the second grading,  by ``standard'' degree $\deg$, for which $\deg(x^a)=0$, $\deg(p_a)=1$, and $\deg(x^*_a)=1$.

The canonical Poisson   and Schouten brackets   on $T^*M[s]$ and  $\Pi T^*M[s]$ are both of weight $-s$ (and degree $-1$).

\subsection{Shifted $\Pinf$- and $\Sinf$-structures.}

Now everything becomes obvious. Fix $(\e,k)\in \Z\times \ZZ$ as above. Let $A$ be a graded commutative superalgebra, i.e. a $\Z\times \ZZ$-graded associative algebra where $ab=ba(-1)^{\at\bt}$.

An   ``$(\e,k)$-shifted homotopy Poisson    structure'' in $A$ is just an  $(\e,k)$-shifted $\Linf$-structure in the sense of Definition~\ref{def.shgrlinf} related with the associative multiplication by the Leibniz rule. It makes sense to elaborate that.

\begin{definition}\label{def.shpoiss}
An   \emph{$(\e,k)$-shifted homotopy Poisson}   structure in $A$ is a sequence of multilinear operations of weights $2-n+k(n-1)$ and parities $\e(n+1)+n$ (i.e. of parity  $n\mod 2$ if $\e=0$ and $1$ for all $n$ if $\e=1$)
 \begin{equation}
    \{-,\ldots,-\}\co \underbrace{A\times \ldots \times A}_{\text{$n$ times}} \to A\,,
  \end{equation}
where $n=0,1,2,\ldots\ $, which \underline{for $\e=0$} are antisymmetric and satisfy
\begin{equation}\label{eq.hjac2}
    \sum_{r+s=n}\sum_{\text{shuffles}}  (-1)^{\b} \{ \{[a_{\s(1)},\ldots,a_{\s(r)}], \ldots , a_{\s(r+s)}\}=0
  \end{equation}
  for all~$n=0,1,2,3,...$, where $(-1)^{\b}=(-1)^{rs}\sign \s (-1)^{\a}$ and $(-1)^{\a}$  is the Koszul sign, and also satisfy
\begin{equation}\label{eq.leibpois}
   \{a_1,\ldots,a_{n-1},bc\}=\{a_1,\ldots,a_{n-1},b\}c+(-1)^{(\at_1+\ldots+\at_{n-1}+n)\bt}b\{a_1,\ldots,a_{n-1},c\}\,,
\end{equation}  
for all $n=1,2,3,...$\,; 
and  \underline{for $\e=1$} are symmetric and  satisfy
\begin{equation}\label{eq.hjac3}
    \sum_{r+s=n}\sum_{\text{shuffles}}  (-1)^{\a} \{ \{[a_{\s(1)},\ldots,a_{\s(r)}], \ldots , a_{\s(r+s)}\}=0
  \end{equation}
  for all~$n=0,1,2,3,...$, where   $(-1)^{\a}$  is the Koszul sign, 
and also satisfy
\begin{equation}\label{eq.shleib}
  \{a_1,\ldots,a_{n-1},bc\}=\{a_1,\ldots,a_{n-1},b\}c+(-1)^{(\at_1+\ldots+\at_{n-1}+1)\bt}b\{a_1,\ldots,a_{n-1},c\}\,,
\end{equation}  
for all $n=1,2,3,...$\,.

If $\e=0$, an  $(\e,k)$-shifted homotopy Poisson  structure is  called a \emph{$k$-shifted $\Pinf$-structure}, and if $\e=1$, it is called  a \emph{$k$-shifted $\Sinf$-structure}.
\end{definition}

 For $k=0$, we have  the usual    $\Pinf$-  and $\Sinf$-structures in the $\Z\times \ZZ$-graded versions.

 If parity is linked with weight, then $k$ must be even for a shifted $\Pinf$ case and odd for a shifted $\Sinf$ case.

\begin{definition}
\label{def.shsinfm}
If $M$ is a graded supermanifold, a \emph{shifted}   $\Pinf$- or \emph{$\Sinf$-structure on $M$} is the corresponding structure in the algebra  $\fun(M)$ (or in the sheaf of local $\fun$ functions).
\end{definition}

\begin{proposition}\label{prop.shiftpoiss}
{\vphantom{}} For a graded supermanifold $M$,

\emph{(1) } A  $k$-shifted $\Pinf$-structure on $M$ is specified by a  master  anti-Hamiltonian 
  \begin{equation*}
    P\in \fun(\Pi T^*M[1-k])
  \end{equation*}
  of weight $2-k$;
 
\emph{ (2) } A $k$-shifted $\Sinf$-structure on $M$ is specified by a  master  Hamiltonian  
  \begin{equation*}
    H\in \fun(T^*M[1-k])
    \end{equation*}
  of weight $2-k$.
\end{proposition}

A ``{master  (anti)Hamiltonian}'' means an odd (resp. even) function on the (anti)cotangent bundle satisfying the master equation $(H,H)=0$ or $\lsch P,P\rsch=0$. It is always regarded as a formal function in the fiber directions. 

\begin{proof}[Proof of Proposition~\ref{prop.shiftpoiss}]
Regardless of shifts, the 
brackets of functions on $M$ for a  $\Pinf$- or $\Sinf$-structure  (``homotopy Poisson'' or ``homotopy Schouten'', resp.) are given by the   higher derived brackets formulas of~\cite{tv:higherder}. 
Consider, for example, the   $\Sinf$ case. We have for  the brackets on $M$ generated by a master Hamiltonian   $H\in \fun(T^*M[1-k])$\,,
\begin{equation}\label{shSinfH}
  \{f_1,\ldots,f_n\}_H=(\ldots(H,f_1),\ldots,f_n)_{|M}\,.
\end{equation}
It only remains to calculate the weights. 
At the RHS, the weights   are calculated on $T^*M[1-k]$. The pullback from $M$ to $T^*M[1-k]$ and the restriction from $T^*M[1-k]$ to $M$ preserve weights. The canonical Poisson bracket on $T^*M[1-k]$ has weight $k-1$. Hence we have the weight of the RHS of~\eqref{shSinfH} as
\begin{equation*}
  \w(f_1)+\ldots + \w(f_n)+(k-1)n+ 2-k =\w(f_1)+\ldots + \w(f_n)+2 -n+k (n-1)
\end{equation*}
as desired.
\end{proof}

The same can be  seen  also  by considering     homological vector fields  corresponding to these  brackets. For example, for an   $\Sinf$-structure on $M$ given by $H\in \fun(T^*M[1-k])$, the corresponding homological vector field   is
\begin{equation}\label{eq.HJ}
  Q_H=\int_M \!Dx\; H\Bigl(x,\der{f}{x}\Bigr)\var{}{f(x)}\,.
\end{equation}
According to Proposition~\ref{prop.homf},  
for the brackets to form a  $k$-shifted $\Linf$-structure, $Q_H$ should be a vector field on $\fun(M)[1-k]$ of weight $+1$,

Note that $f$ as a `point' of $\fun(M)[1-k]$ is an even function on $M$ of weight $1-k$ (which gives  weight $0$ in $\fun(M)[1-k]$). Count the weights in~\eqref{eq.HJ}. The variational derivative w.r.t. $f(x)$ of weight $1-k$ has weight $-1+k$. Therefore, the condition  $\w(Q_H)=+1$ is equivalent to $\w\left(H\Bigl(x,\der{f}{x}\Bigr)\right)=2-k$. But note that because $\w(f)=1-k$, the weight of $H\Bigl(x,\der{f}{x}\Bigr)$  on $M$ is exactly $\w^{[1-k]}(H)$, i.e.  the weight of $H$ regarded as a function on   $T^*M[1-k]$. Hence  $\w(Q_H)=+1$ is equivalent to $\w^{[1-k]}(H)=2-k$, as claimed.

\begin{remark}
   $\Pinf$- and $\Sinf$-structures  (in the purely super setting)   were introduced in~\cite{tv:higherder, tv:higherderarb}  under the names `homotopy Poisson algebra' and `homotopy Schouten algebra'.     $\Pinf$-algebras with this name were first defined in~\cite{cattaneo-felder:relative2007}  (in a $\ZZ$-graded setting).  Independently, I used the terminology with $\Pinf$- and also $\Sinf$ in the super setting in~\cite{tv:nonlinearpullback}  following on the notations      $P$ and  $S$ of~\cite{tv:graded} for  ``even Poisson'' and ``odd Poisson''.    Homological vector fields for $\Pinf$- and $\Sinf$-structures were found in~\cite{tv:nonlinearpullback}. 
\end{remark}

 \section{Graded version of microformal morphisms}
 
 Thick or microformal morphisms between graded manifolds can be easily adapted to the graded context so for the corresponding non-linear pullbacks to be compatible with grading of functions. We have mentioned that in~\cite{tv:microformal}, but below we make it explicit.

 \subsection{Construction}
 
Recall that a \emph{thick} or \emph{microformal morphism} $\F\co M_1\tto M_2$ (where $M_1$ and $M_2$ are supermanifolds) is a ``framed'' formal canonical relation $T^*M_1\dashrightarrow T^*M_2$ in the sense that it is specified in local coordinates by a generating function $S=S(x,q)$ regarded as part of the structure, where $x^a$ are coordinates on $M_1$ and $q^i$ are momentum variables on $T^*M_2$  conjugate to local coordinates $y^i$ on $M_2$, which is a formal power series
\begin{equation}\label{eq.s}
  S(x,q)=S^0(x) + \f^i(x)q_i + \frac{1}{2}S^{ij}(x)q_jq_i+\ldots \,
\end{equation}
(even,  in the sense of parity). A generating function~\eqref{eq.s} specifies the conjugate momenta $p_a$ and the coordinates $y^i$ as functions of $x^a,q_i$ via
\begin{equation}\label{eq.relat}
  p_a=\der{S}{x^a}  \quad \text{and} \quad y^i=(-1)^{\itt}\der{S}{q_i}\,.
\end{equation}
Here  $S^0=S^0(x)$ is a well-defined function on $M_1$, the first-order term of~\eqref{eq.s} defines an ordinary map $\f\co M_1\to M_2$, called the \emph{support} of a thick morphism $\F$,  by $y^i=\f^i(x)$, and the whole power series~\eqref{eq.s} undergoes a particular transformation under changes of local coordinates, so to make everything coordinate-independent. See details in~\cite{tv:nonlinearpullback, tv:microformal}.  

The key motivation for this notion is the possibility to define a pullback $\Phi^*$ as a non-linear formal map of the infinite-dimensional supermanifolds of functions:
\begin{equation}\label{eq.phistar}
  \Phi^*\co \funn(M_2)\to \funn(M_1)\,.
\end{equation}
(The `points' of $\funn(M)$ are even functions on $M$.) Here is the formula:
\begin{equation}\label{eq.pullb}
  \Phi^*\co g(y)\mapsto f(x)=g(y)+S(x,q)-y^iq_i\,, 
\end{equation}
where $y^i$ and $q_i$ are determined from the equations
\begin{equation}\label{eq.pullbaux}
   y^i=(-1)^{\itt}\der{S}{q_i}(x,q)   \quad \text{and} \quad  q_i=\der{g}{y^i}(y)\,. 
\end{equation}
This uniquely defines $f(x)$ as a formal power series in the derivatives of $g$ of increasing order evaluated at $y=\f(x)$\,: 
\begin{equation}\label{eq.pullexpl}
  f(x)= S^0(x) + g(\f(x)) + \frac{1}{2}S^{ij}(x)\, \p_jg(\f(x))\,\p_ig(\f(x)) + \ldots 
\end{equation}
(See also~\cite{tv:operovermap} and \cite{swerdlow:graph2025}.)

(There is a parallel notion of ``odd'' thick morphisms $M_1\oto M_2$ based on formal canonical relations $\Pi T^*M_1 \dashrightarrow\Pi T^*M_1$ specified by odd generating functions $S(x,y^*)$, which gives a formal map $\pfunn(M_2)\to \pfunn(M_1)$, and also a ``quantum'' version that gives a ``quantum pullback'' as a particular formal Fourier integral operator.)

Suppose now $M_1$ and $M_2$ are graded supermanifolds. Therefore the spaces of functions are graded vector spaces that we can also consider as graded supermanifolds. Taking into account possible shifts and parity reversion, we can consider the graded supermanifolds of the form $\funn(M)[s]$ and $\pfunn(M)[s]$, $s\in \ZZ$. `Points' of $\funn(M)[s]$ and $\pfunn(M)[s]$ are even (resp., odd) functions of weight $s$ on $M$. Here  $M=M_1,M_2$.

Consider, as above, shifted (anti)cotangent bundles.  Thus  for   $M_1$  we have
\begin{equation*}
  T^*M_1[s]\quad \text{and} \quad \Pi T^*M_1[s]\,,
\end{equation*}
meaning by definition that the weights are counted as  $\w^{[s]}(p_a)=-w^a+s$, and $\w^{[s]}(x^*_a)=-w^a+s$, where $\w^{[s]}(x^a)=\w(x^a)=w^a$.  Similarly for $M_2$, where $\w(y^i)=w^i$, etc. 

Let the generating function~\eqref{eq.s} of a thick morphism $\Phi\co M_1\tto M_2$ have weight $s$ with respect to  weights shifted this way. Then
\begin{equation}\label{eq.wrelat}
  \w^{[s]}\left(\der{S}{x^a}\right)= s-\w(x^a)=s-w^a \equiv \w^{[s]}(p_a)   
\end{equation}
and
\begin{equation}\label{eq.wrelat}
  \w^{[s]}\left(\der{S}{q_i}\right)= s -  \w^{[s]}(q_i)=       s-(-w^i+s)=w^i\equiv \w(y^i)\,.
\end{equation}
Therefore, equations~\eqref{eq.relat} agree with weights, and the generating function $S$ specifies a well-defined  formal  canonical relation $T^*M_1[s]\dashrightarrow T^*M_2[s]$ of   shifted cotangent bundles. 

Also,  if $g(y)$ has weight $s$, then the second equation in~\eqref{eq.pullbaux} will agree with $\w^{[s]}(q_i)=-w_i+s$, hence  in equation~\eqref{eq.pullb} the term $y^iq_i$ will have weight $s$, and therefore the whole right-hand side of~\eqref{eq.pullb} will be homogeneous of weight $s$. Hence the function $f(x)$ defined by these equations will be of weight $s$ as well. 

\begin{definition}\label{def.shthick}
We   call a thick morphism $M_1\tto M_2$ between graded supermanifolds with a generating function $S$ of weight $s$ (with respect to shifted weights), an \emph{$s$-shifted thick morphism}. 
\end{definition}

Hence we obtain the following statement.

\begin{theorem}
  For an $s$-shifted thick morphism $\Phi\co M_1\tto M_2$, its   pullback $\Phi^*$ by the usual formulas~\eqref{eq.pullb} and~\eqref{eq.pullbaux}   is a well-defined formal map of graded supermanifolds
  \begin{equation}\label{eq.shphistar}
  \Phi^*\co \funn(M_2)[s]\to \funn(M_1)[s]\,.
\end{equation}
\end{theorem}

 There is a similar notion of  \emph{$s$-shifted odd thick morphisms} $M_1\oto M_2$ for graded supermanifolds, based on shifted anticotangent bundles, $\Pi T^*M_1[s]$ and  $\Pi T^*M_2[s]$, and a similar theorem for the corresponding pullbacks, which will give $\pfunn(M_2)[s]\to \pfunn(M_1)[s]$\,.
 
 \begin{remark}
   As A.~Swerdlow  informed me, he has independently arrived at the same notion of shifted thick morphisms as in Definition~\ref{def.shthick}. 
 \end{remark}
 
 \subsection{Application to shifted Poisson brackets}
 
Suppose $M_1$ and $M_2$ are equipped with $k$-shifted $\Sinf$-structures in the sense of Definitions~\ref{def.shpoiss} and \ref{def.shsinfm}.
They are specified by master Hamiltonians $H_1 \in \fun(T^*M_1)[1-k]$, $H_2\in \fun(T^*M_2)[1-k]$ of weight  $2-k$, which give the corresponding Hamilton-Jacobi homological vector fields $Q_{H_1}$ and $Q_{H_2}$ of weight $+1$, as we discussed in Sec.~\ref{sec.shpinfsinf}.
 
 Let $S=S(x,q)$ be the generating function of an $s$-shifted thick morphism $\Phi\co M_1\tto M_2$. It should have weight $s$ with respect to $s$-shifted cotangent bundles. Let $s=1-k$, and consider the Hamilton-Jacobi equation
 \begin{equation}\label{eq.HJ}
   H_1\Bigl(x,\der{S}{x}\Bigr)=H_2\Bigl((-1)^{\itt}\der{S}{q_i},q\Bigr)\,.
 \end{equation}
 Observe that $\w^{[s]}(\der{S}{x^a})=\w^{[s]}(S)-w^a=s-w^a=1-k-w^a$, which coincides with $\w^{[1-k]}(p_a)=-w^a+1-k$. Likewise, 
 $\w^{[s]}(\der{S}{q_i})=\w^{[s]}(S)-\w^{[s]}(q_i)=s-(-w^i+s)=w^i$, which agrees with $\w^{[s]}(y^i)=\w(y^i)=w^i$. Hence equation~\eqref{eq.HJ} is compatible with our shifted grading in the cotangent bundles. 
 
 \begin{definition}
   An $s$-shifted thick morphism $\Phi\co M_1\tto M_2$ is a \emph{shifted $\Sinf$-thick morphism} for $k$-shifted $\Sinf$-structures on $M_1$ and $M_2$, where $s=1-k$, if its generating function $S(x,q)$ satisfies~\eqref{eq.HJ}. 
 \end{definition}
 
We immediately arrive at the following theorem.
\begin{theorem}
  The pullback $\Phi^*$ by shifted $\Sinf$-thick morphism $\Phi\co M_1\tto M_2$ for $k$-shifted $\Sinf$-structures on graded supermanifolds  $M_1$ and $M_2$ is a well-defined formal $Q$-morphism 
  \begin{equation}\label{eq.shmor}
  \Phi^*\co \funn(M_2)[1-k]\to \funn(M_1)[1-k]\,.
\end{equation}
  of graded infinite-dimensional supermanifolds, and thus specifies an $\Linf$-morphism of the shifted  bracket structures.
\end{theorem}
 
 A similar statement holds for   shifted odd thick morphisms and shifted $\Pinf$-structures. 
 


\def\cprime{$'$} \def\cprime{$'$} \def\cprime{$'$} \def\cprime{$'$}
  \def\cprime{$'$} \def\cprime{$'$} \def\cprime{$'$} \def\cprime{$'$}
  \def\cprime{$'$} \def\cprime{$'$} \def\cprime{$'$} \def\cprime{$'$}
  \def\cprime{$'$}

\end{document}